\documentclass[10pt]{article}

\usepackage{graphicx}

\usepackage{amssymb}

\usepackage{amsthm}

\usepackage{epstopdf}

\usepackage{amsmath}
\usepackage{amssymb}
\usepackage{amsfonts}

\def\phi{{\varphi}}

\newcommand{\CO}[2]{ \left\langle #1 , #2 \right\rangle}

\DeclareSymbolFont{AMSb}{U}{msb}{m}{n}
\DeclareMathSymbol{\N}{\mathbin}{AMSb}{"4E}
\DeclareMathSymbol{\Z}{\mathbin}{AMSb}{"5A}
\DeclareMathSymbol{\R}{\mathbin}{AMSb}{"52}
\DeclareMathSymbol{\Q}{\mathbin}{AMSb}{"51}
\DeclareMathSymbol{\I}{\mathbin}{AMSb}{"49}
\DeclareMathSymbol{\C}{\mathbin}{AMSb}{"43}

\def\be{\begin{equation}}
\def\ee{\end{equation}}
\def\ber{\begin{eqnarray}}
\def\eer{\end{eqnarray}}

\def\beq{\begin{equation}}
\def\eeq{\end{equation}}

\def\Z{{\mathbb{Z}}}

\newcommand{ \pOm}{\partial \Omega}

\begin{document}

\addtolength{\textheight}{0 cm} \addtolength{\hoffset}{0 cm}
\addtolength{\textwidth}{0 cm} \addtolength{\voffset}{0 cm}

\newenvironment{acknowledgement}{\noindent\textbf{Acknowledgement.}\em}{}

\setcounter{secnumdepth}{5}

 \newtheorem{proposition}{Proposition}[section]
\newtheorem{theorem}{Theorem}[section]
\newtheorem{lemma}[theorem]{Lemma}
\newtheorem{coro}[theorem]{Corollary}
\newtheorem{remark}[theorem]{Remark}
\newtheorem{extt}[theorem]{Example}
\newtheorem{claim}[theorem]{Claim}
\newtheorem{conj}[theorem]{Conjecture}
\newtheorem{definition}[theorem]{Definition}
\newtheorem{application}{Application}
\newtheorem*{thm*}{Theorem A}
\newtheorem{corollary}[theorem]{Corollary}
\title{Multiplicity results for elliptic problems with super-critical concave and convex nonlinearties 
\footnote{Abbas Moameni is  pleased to acknowledge the support of the  National Sciences and Engineering Research Council of Canada.
}}
\author{Najmeh Kuhestani\footnote{Department of Mathematics, Kharazmi University.
 School of Mathematics and Statistics,
Carleton University,
Ottawa, Ontario, Canada.} \quad  Abbas Moameni \footnote{School of Mathematics and Statistics,
Carleton University,
Ottawa, Ontario, Canada,
momeni@math.carleton.ca} }

\date{}

\maketitle

\vspace{3mm}

\begin{abstract} We shall prove a multiplicity result for  semilinear elliptic problems with a super-critical  nonlinearity  of the form,
\begin{equation}\label{con-c}
\left \{
\begin{array}{ll}
-\Delta u =|u|^{p-2} u+\mu |u|^{q-2}u, &   x \in \Omega\\
u=0, &  x \in \partial \Omega
\end{array}
\right.
\end{equation}
where $\Omega\subset \R^n$ is a bounded domain with $C^2$-boundary and  $1<q< 2<p.$ As a consequence of our results we shall show that, for each $p>2$, there exists $\mu^*>0$ such that for each $\mu \in (0, \mu^*)$   problem (\ref{con-c}) has a sequence of solutions with a negative energy.  This result was already known for the subcritical values of $p.$  In this paper, we shall extend it to the supercritical values of $p$ as well.
 Our methodology is based on  a new variational principle established by one of the authors  that allows one to deal   with problems beyond the usual locally compactness structure.
\end{abstract}

\section{Introduction}

In this paper we consider the semilinear elliptic problem
\begin{equation}\label{eq}
\left\{\begin{array}{ll}
-\Delta u =u | u|^{p-2}+\mu u | u| ^{q - 2}, &  \mbox{ in } \Omega, \\
u = 0, &   \mbox{ on } \pOm,
\end {array}\right.
\end{equation}
where $\Omega \subset \mathbb{R}^{n}$ is a bounded domain with $C^{2}$-boundary,  $\mu \in \mathbb{R}^+$ and   $1 < q < 2<p. $
This problem has received a lot of attention since being first investigated by Ambrosetti, Brezis and Cerami in \cite{Ambro}.   Using the method of sub-super solutions, it is proved in \cite{Ambro} that there exists $\Lambda > 0$ such that (\ref{eq}) has a positive
solution $\underline{u}_{\mu}$ for $0 < \mu \leq \Lambda$. The importance of their results lies in the fact that  $p$ can be arbitrarily large. If in addition $p < 2^{*} := 2n/(n-2),$ then solutions of (\ref{eq}) correspond
to critical points of the functional
\begin{equation}
I(u) = \frac{1}{2}\int_{\Omega}| \nabla u |^{2} -\frac{1}{p} \int_{\Omega} | u | ^{p} dx - \frac{\mu}{q} \int_{\Omega} | u | ^{q} dx,
\end{equation}
defined on $H_{0}^{1}(\Omega)$, and hence variational methods may be  applied. In this  case a second positive solution $\overline{u}_{\mu}$ exists for $0 < \mu \leq \Lambda $ as  shown in \cite{Ambro}, Theorem 2.3. Moreover, there exists $\Lambda > 0$
such that for every $0 < \mu < \Lambda$ problem (\ref{eq}) has infinitely many solutions $\{\underline{u}_{\mu, j}\}_{j\in \mathbb{N}}$ satisfying $I(\underline{u}_{\mu, j}) < 0$, and there exist infinitely many solutions $\{\overline{u}_{\mu, j}\}_{j\in \mathbb{N}}$ satisfying $I(\overline{u}_{\mu, j}) > 0$.  In fact, they showed that there exists an additional pair of solutions (which can change sign) for all $0<\mu<\mu^{*}$ with $\mu^{*}$ possibly smaller than $\Lambda$ (see also Ambrosetti, Azorero
and Peral \cite{Azorero} and  references therein).  Their method relied on the standard methods in the  critical point theory.
\\
Over the years, the study for the number of positive solutions  were furthered by many authors including  \cite{Adimurthi,Damascelli, Ouyang, Tang}.  It was indeed established that   if  $1 < q < 2 < p \leq   2^{*}$ then  there exists $\mu^{*} > 0$ such that for $0<\mu<\mu^{*}$, there are exactly two positive solutions of (\ref{eq}), exactly one positive solution for $\mu = \mu^{*}$ and no positive solution exists for $\mu > \mu^{*}$, when $\Omega$ is the unit ball in $\mathbb{R}^{n}$. 
\\
In \cite{Ambro} and \cite{Garcia} the existence of solutions with negative energy has also
been proved in the critical case p = 2* provided $\mu> 0$ is small enough. 

Also, Bartsch and Willem \cite{Willem} showed that   for the subcritical case $\mu^{*} = \infty$ and  $I(\underline{u}_{\mu, j}) \rightarrow 0$  as  $j \rightarrow\infty$.
 In addition they proved that a sequence of solutions $\{\overline{u}_{\mu, j}\}$ with a positive energy  also exists for $\mu \leq 0$. Furthermore, Wang \cite{Wang}
proved that the solutions $\underline{u}_{\mu, j}$ not only tend to $0$ energetically but also uniformly on $\Omega$. Wang
even dealt with more general classes of nonlinear functions $f_{\mu}(u)$ instead of just $u | u| ^{p - 2} +\mu u | u| ^{q - 2}$. The
variational structure and the oddness of the nonlinearity, however, are essential to obtain
infinitely many solutions $\{\underline{u}_{\mu, j}\}$ and $\{\overline{u}_{\mu, j}\}$ for the subcritical case.\\

Our main objective in this paper is to prove multiplicity results without imposing any growth condition on the nonlinearity $u|u|^{p-2}.$
We shall now state our result in this paper regarding positive solutions of (\ref{eq}). 
\begin{theorem}\label{main1}
Assume that $1 < q < 2 < p$. Then  there exists $\mu ^{*} > 0$ such that for each $\mu \in (0,\mu ^{*})$ problem (\ref{eq}) has at least one  positive  solution $\overline{u} \in W^{2, n}(\Omega)$ with a negative energy. 
\end{theorem}
This result, however,  is already known in \cite{Ambro}.  Here we shall provide a different approach based on variational principles on convex closed sets.  The next result concerns with the multiplicity of solutions for the super-critical case.  The next theorem addresses the multiplicity result for the super-critical case.

\begin{theorem}\label{main2}
Assume that $1 < q < 2 < p$. Then there exists $\mu ^{*} > 0$ such that for each $\mu \in (0,\mu ^{*})$ problem (\ref{eq}) has infinitely many distinct nontrivial solutions with a negative energy. 
 \end{theorem} 
 As there is no upper bound for $p$ in Theorem \ref{main2},  thus, this theorem will be  an extension of a  similar result by Ambrosetti-Brezis-Cerami \cite{Ambro}   to  the supercritical case.
\begin{remark} Note that  the term $u | u|^{p-2}$  can be substituted by any super-linear  odd function $f$ that behaves like $f(u)=u | u|^{p-2}$ around $u=0$ and around $u=+\infty.$ The oddness of $f$ is not required in Theorem \ref{main1}, however, f has to be positive on $(0,\infty).$  We would also like to remark that the parameter $\mu^*$ is the same in both Theorems \ref{main1} and \ref{main2}. It is also worth nothing that, there exists $\Lambda \in (0, \infty)$ such that  problem  (\ref{eq}) does not have any solution for $\lambda >\Lambda$ (See Theorem 2.1 in \cite{Ambro}).
\end{remark}

We shall be proving Theorems \ref{main1} and \ref{main2} by making use of  a new abstract variational principle established recently in \cite{Mo5, Mo1} (see also \cite{Mo,  Mo2} for some new variational principles and \cite{ACL} for an application in super-critical Neumann problems).
To be more specific,
let $V$ be a reflexive Banach space, $V^*$ its topological dual  and let $K$ be a convex and weakly closed subset of $V$.
Assume that  $\Psi : V \rightarrow \mathbb{R} \cup \{+\infty\}$ is a proper, convex, lower semi-continuous function and G\^ateaux differentiable on K (with G\^ateaux derivative $D\Psi(u)$). The restriction of $\Psi$ to $K$ is denoted by $\Psi_K$ and defined by
\begin{eqnarray}
\Psi_K(u)=\left\{
  \begin{array}{ll}
      \Psi(u), & u \in K, \\
    +\infty, & u \not \in K.
  \end{array}
\right.
\end{eqnarray}
For a given functional $\Phi \in C^{1}(V, \mathbb{R})$, consider the functional $I_K: V \to (-\infty, +\infty]$ defined by
 \begin{eqnarray*}
 I_K(u):= \Psi_K(u)-\Phi(u).
 \end{eqnarray*}
According to Szulkin \cite{szulkin}, we have the following definition for critical points of $I_K$.

\begin{definition}\label{ddd}
A point $u\in  V$ is said to be a critical point of $I_K$ if $I_K(u) \in \mathbb{R}$ and if it satisfies the following inequality
\begin{equation}
 \CO{D \Phi(u)}{ u-v} + \Psi_K(v)- \Psi_K(u) \geq 0, \qquad \forall v\in V.
\end{equation}
\end{definition}

We shall now recall the following variational principle established recently in \cite{Mo5}.
\begin{theorem}\label{con22}
Let $V$ be a reflexive Banach space and $K$ be a convex and weakly closed subset of $V$. Let $\Psi : V \rightarrow \mathbb{R}\cup \{+\infty\}$ be a convex, lower semi-continuous function which is G\^ateaux differentiable on $K$ and $\Phi \in C^{1}(V, \mathbb{R})$. If the following two assertions hold:
\begin{enumerate}
\item[$(i)$] The functional $I_K: V \rightarrow \mathbb{R} \cup \{+\infty\}$ defined by  $I_K(u):= \Psi_K(u)-\Phi(u)$ has a critical point $u_{0}\in V$ as in Definition \ref{ddd}, and; 

\item[$(ii)$] there exists $v_{0}\in K$ such that $D \Psi(v_{0}) = D \Phi(u_{0})$.
\end{enumerate} 
Then $u_{0}\in K$ is a solution of the equation 
 \begin{equation} \label{equ1} D \Psi(u) = D \Phi(u). \end{equation}
 \end{theorem}
 
For the convenience of the reader, by choosing the functions $\Psi,$ $\Phi$ and the convex set $K$ in lines with problem (\ref{eq}),   we shall provide a proof to a particular case of Theorem \ref{con22} applicable to this problem.\\

In the next section we shall recall some preliminaries from convex analysis, critical point theory and Elliptic regularity theory. Section 3 is devoted to the proof of Theorems \ref{main1} and \ref{main2}.

\section{Preliminaries} \label{convex}

In this section we recall some important definitions and results from convex analysis \cite{Ek-Te} and partial differential equations \cite{G-T}.

Let $V$ be a  real Banach  space and $V^*$ its topological dual  and let $\langle .,. \rangle $ be the pairing between $V$ and $V^*.$
The weak topology on $V$ induced by $\langle .,. \rangle $ is denoted by $\sigma(V,V^*).$  A function $\Psi : V \rightarrow \mathbb{R}$ is said to be weakly lower semi-continuous if
\[\Psi(u) \leq \liminf_{n\rightarrow \infty} \Psi(u_n),\]
for each $u \in V$ and any sequence ${u_n} $ approaching $u$ in the weak topology $\sigma(V,V^*).$
Let $\Psi : V \rightarrow \mathbb{R}\cup \{\infty\}$ be a proper (i.e. $Dom(\Psi)=\{v \in V; \, \Psi(v)< \infty\} \neq \emptyset$) convex  function. The subdifferential $\partial \Psi $ of $\Psi$
is defined  to be the following set-valued operator: if $u \in Dom (\Psi),$ set
\[\partial \Psi (u)=\{u^* \in V^*; \langle u^*, v-u \rangle + \Psi(u) \leq \Psi(v) \text{  for all  } v \in V\}\]
and if $u \not \in Dom (\Psi),$ set $\partial \Psi (u)=\varnothing.$ If $\Psi$ is G\^ateaux differentiable at $u,$ denote by $D \Psi(u)$ the derivative of $\Psi$ at $u.$ In this case  $\partial \Psi (u)=\{ D  \Psi(u)\}.$\\
 
 Let $I$ be a function on $V$ satisfying the following hypothesis:
 \\
\noindent
{(H):}$\quad$ \emph{$I= \Psi - \Phi  $, where $\Phi\in C^1(V,\mathbb{R})$ and $\Psi: V\rightarrow (-\infty, +\infty]$ is proper, convex and lower semi-continuous.
}
\\
\begin{definition}\label{dd}A point $u\in V$ is said to be a critical point of $I$ if $u\in Dom(\Psi)$ and if it satisfies
the inequality
\begin{equation}\label{critical point}
 \CO{D \Phi(u)}{ u-v} + \Psi(v)- \Psi(u) \geq 0, \qquad \forall v\in V.
\end{equation}
\end{definition}
Note that a function satisfying (\ref{critical point}) is indeed a solution of the inclusion $D \Phi(u) \in \partial \Psi(u)$.

\begin{proposition}\label{minimum}
If $I$ satisfies (H), each local minimum is necessarily a critical point of $I$.
\end{proposition}

\begin{proof} Let $u$ be a local minimum of $I$. Using convexity of $\Psi$, it follows that for all small $t > 0$,
\begin{align*}
0 \leq I\left((1 -t)u + tv\right) - I(u) &= \Phi\left(u + t(v -u)\right) - \Phi(u) + \Psi\left((1 -t)u + tv\right) - \Psi(u)\\
&\leq \Phi\left(u + t(v -u)\right) - \Phi(u) + t\left(\Psi(v) - \Psi(u)\right). 
\end{align*}
Dividing by $t$ and letting $t\rightarrow 0^+$ we obtain (\ref{critical point}).
\end{proof} 

The critical point theory for functions of the type $(H)$ was established by Szulkin in \cite{szulkin}.  According to \cite{szulkin}, say  that $I$ satisfies the  compactness condition of Palais-Smale type provided,
\\
\noindent
{(PS):}$\quad$ \emph{If $ \{u_{n}\}$ is a sequence such that $I(u_{n})\rightarrow c \in \mathbb{R}$ and 
\begin{align} \label{ps}
\langle D\varphi (u_{n}), u_{n} - v \rangle + \psi (v) - \psi(u_{n}) \geq -\epsilon_{n} \| v - u_{n}\|_{V}\qquad \forall v \in V,
\end{align}
where $\epsilon_{n}\rightarrow 0$, then $\{u_{n}\}$  possesses a convergent subsequence}. 
\\

In the following we recall  an important result  about critical points of even functions of the type $(H)$. We shall begin with some preliminaries. Let $\Sigma$ be the of all symmetric subsets of $V\setminus \{0\}$ which are closed in $V$. A nonempty set $A \in \Sigma$ is said to have \emph{genus k} (denoted $\gamma(A) = k$) if $k$ is the smallest integer with the property that there exists an odd continuous mapping $h : A\rightarrow \mathbb{R}^{k} \setminus\{0\}$. If such an integer does not exist, $\gamma(A) = \infty$. For the empty set $\emptyset$ we define $\gamma(\emptyset) = 0$.

\begin{proposition}\label{genus}
Let $A \in \Sigma$. If $A$ is a homeomorphic to $S^{k -1}$ by an odd homeomorphism, then $\gamma(A) = k$. 
\end{proposition}
Proof and a more detailed discussion of the notion of genus can be found in \cite{Rabinowitz 1} and \cite{Rabinowitz 2}.
\\
Let $\Theta$ be the collection of all nonempty closed and bounded subsets of $V$. In $\Theta$ we introduce the Hausdorff metric distance [\ref{r18},$\S15, VII$], given by 
\begin{align*}
dist(A, B) =\max\{\sup_{a\in A}d(a, B), \sup_{b\in B}d(b, A)\}.
\end{align*}
The space $(\Theta, dist)$ is complete [\ref{r18},$\S29, IV$]. Denote by $\Gamma$ the sub-collection of $\Theta$ consisting of all nonempty compact symmetric subsets of $V$ and let  
\begin{align}\label{Gam}
\Gamma_{j} = cl\{A\in \Gamma : 0\notin A, \gamma(A)\geq j \}
\end{align}
($cl$ is the closure in $\Gamma$). It is easy to verify that $\Gamma$ is closed in $\Theta$, so $(\Gamma, dist)$ and $(\Gamma_{j}, dist)$ are complete metric spaces.
The following Theorem is proved in \cite{szulkin}.
\begin{theorem}\label{critical theorem}
Suppose that $I : V\rightarrow (-\infty, +\infty]$ satisfies (H) and {(PS)}, $I(0) = 0$ and $\Phi$, $\Psi$ are even. Define 
\begin{align*}
c_{j} = \inf_{A\in \Gamma_{j}} \sup_{u \in A} I(u).
\end{align*}
If $-\infty < c_{j} < 0$ for $j = 1, ... ,k$, then $I$ has at least $k$ distinct pairs of nontrivial critical points by   means of Definition \ref{dd}.
\end{theorem}

We shall now recall some notations and results from the theory of  Sobolev spaces and Elliptic regularity required in the sequel. Here is the general Sobolev embedding 
 theorem in  $W^{k, p}(\Omega)$ (see Lemma 7.26 in  \cite{G-T}). 
\begin{theorem}\label{imbedding theorem}
Let $\Omega$ be a bounded $C^{0, 1}$ domain in $\mathbb{R}^{n}$. Then,
\begin{enumerate}
\item[ $(i)$] If $kp < n$, the space $W^{k, p}(\Omega)$ is continuously imbedded in $L^{t^*}(\Omega)$, $t^* = \frac{np}{(n - kp)}$, and compactly imbedded in $L^{q}(\Omega)$ for any $q < t^{*}$.
\item[ $(ii)$] If $0 \leq m < k - \frac{n}{p} < m + 1$, the space $W^{k, p}(\Omega)$ is continuously imbedded in $C^{m, \alpha}(\overline{\Omega})$, $\alpha =  k - \frac{n}{p} - m$, and compactly imbedded in $C^{m, \beta}(\overline{\Omega})$ for any $\beta < \alpha$.
\end{enumerate} 
\end{theorem}

The following inequality is is proved in (\cite{G-T}, Lemma 9.17).
\begin{lemma}\label{elliptic theorem for p} Let $\Omega$ be a bounded $C^{1,1}$ domain in $\R^n$
and the operator $L =  a^{ij}(x)D_{ij}u + b^{i}(x)D_{i}u + c(x)u$ be strictly elliptic in $\Omega$ with coefficients $a^{ij}\in C(\Omega)$, $b^{i}, c\in L^{\infty}(\Omega)$, with $i, j = 1, ..., n$ and $c \leq 0$. Then there exists a positive constant $C$ (independent of u) such that  
\begin{equation*} 
\| u \|_{W^{2, p}(\Omega)}\leq  C\| Lu \|_{L^{p}(\Omega)}, 
\end{equation*} 
for all $u\in W^{2, p}(\Omega)\cap W_{0}^{1, p}(\Omega)$, $1 < p < \infty$.
\end{lemma}
Here is a direct consequence of Lemma \ref{elliptic theorem for p}.
\begin{corollary}\label{p-est} Let $\Omega$ be a bounded $C^{1,1}$ domain in $\R^n.$ Assume that $p\geq 2.$ Then there exist constants $\Lambda_1$ and  $\Lambda_2$ such that 
\begin{equation*} 
\Lambda_1 \| u \|_{W^{2, p}(\Omega)}\leq  \|\Delta u\|_{L^p(\Omega)} \leq \Lambda_2 \| u \|_{W^{2, p}(\Omega)}, 
\end{equation*} 
for all $u\in W^{2, p}(\Omega)\cap H_0^1(\Omega).$
\end{corollary}
\textbf{Proof.} Since $p \geq 2,$ it is easily seen that $W^{2, p}(\Omega)\cap H_0^1(\Omega)=W^{2, p}(\Omega)\cap W_{0}^{1, p}(\Omega).$  Thus, the existence of $\Lambda_1$ follows from Lemma \ref{elliptic theorem for p}.  The existence of $\Lambda_2$ follows from the definition of the Sobolev space $W^{2, p}(\Omega).$ \hfill $\square$\\

\section{Proofs and further comments}

We shall need some preliminary results before proving Theorems \ref{main1} and Theorem \ref{main2} in this section.
We shall  consider the Banach space $V = H_0^1(\Omega) \cap L^p(\Omega)$  equipped with the following norm
\begin{align*}
\| u \| := \| u \|_{H_0^1(\Omega)} + \| u \|_{L^p(\Omega)}. 
\end{align*}
Let  $I : V\rightarrow \mathbb{R}$ be the Euler-Lagrange functional corresponding to (\ref{eq}),

\begin{equation*} 
I(u) = \frac{1}{2} \int_{\Omega} | \nabla u | ^{2} dx - \dfrac{1}{p} \int_{\Omega} | u | ^{p} dx - \dfrac{\mu}{q} \int_{\Omega} | u | ^{q} dx .
\end{equation*}
To make use of Theorem \ref{con22}, we shall first define  the function $\Phi : V\rightarrow \mathbb{R}$ by 
\begin{align*}
\Phi (u) = \dfrac{1}{p} \int_{\Omega} | u | ^{p} dx+ \dfrac{\mu}{q} \int_{\Omega} | u | ^{q} dx ,
\end{align*}
Note that $\Phi \in C^1(V;\R).$ 
 Define 
 $\Psi : V \rightarrow \R$ by 
\begin{equation*} 
\Psi(u) = \frac{1}{2} \int_{\Omega} | \nabla u | ^{2} dx.
\end{equation*}
The restriction of $\Psi$ to  a convex and weakly closed subset $K$ of $V$ is denoted by $\Psi_{K}$ and defined by 
\begin{eqnarray}
\Psi_K(u)=\left\{
  \begin{array}{ll}
      \Psi(u), & u \in K, \\
    +\infty, & u \not \in K,
  \end{array}
\right.
\end{eqnarray}
Finally, let us introduce  the functional $I_K: V \to (-\infty, +\infty]$ defined by
 \begin{eqnarray}\label{20}
 I_K(u):= \Psi_K(u)-\Phi(u),
 \end{eqnarray}
 which is of the form $(H).$ Note that 
 $I_K$ is indeed  the Euler-Lagrange functional corresponding to (\ref{eq}) restricted to  $K$. Here is a simplified version of Theorem \ref{con22} applicable to problem (\ref{eq}).

 \begin{theorem}\label{con2}
Let $V=H_0^1(\Omega) \cap L^p(\Omega)$, and let  and $K$ be a convex and weakly closed subset of $V$. If the following two assertions hold:
\begin{enumerate}
\item[$(i)$] The functional $I_K: V \rightarrow \mathbb{R} \cup \{+\infty\}$ defined in (\ref{20})  has a critical point $\bar u\in V$ as in Definition \ref{dd}, and; 

\item[ $(ii)$] there exists $\bar v\in K$ such that $-\Delta \bar v = D \Phi(\bar u)=\bar u |\bar  u|^{p-2}+\mu \bar u | \bar u| ^{q - 2}.$
\end{enumerate} 
Then $\bar u\in K$ is a solution of the equation 
 \begin{equation} \label{equ1} -\Delta  u = u |  u|^{p-2}+\mu u | u| ^{q - 2}. \end{equation}
 \end{theorem}
\textbf{Proof.} Since $\bar u$ is a critical point of $I(u)=\Psi_K(u)-\Phi(u),$
it follows from Definition \ref{dd} that 
\begin{equation}\label{ineq0}
\Psi_K(v)-\Psi_K(\bar u)\geq \langle D \Phi(\bar u),v-\bar u\rangle,\quad \forall v\in V,
\end{equation}
where $\langle D \Phi(\bar u),v-\bar u\rangle=\int_\Omega D \Phi(\bar u)(v-\bar u)\, dx.$
It follows from $(ii)$ in the theorem that  $-\Delta \bar v=D\Phi(\bar u)$. Thus, it follows from  inequality  (\ref{ineq0}) with  $v=\bar v$ that 
\begin{equation}\label{ineq2}
\frac{1}{2} \int_{\Omega} | \nabla \bar v | ^{2} dx-\frac{1}{2} \int_{\Omega} | \nabla \bar u | ^{2} dx\geq \int_\Omega \nabla \bar v.\nabla (\bar v-\bar u)\, dx.
\end{equation}
On the other hand, it follows from the convexity of $\Psi$ that 
\begin{equation}\label{ineq22}
\frac{1}{2} \int_{\Omega} | \nabla \bar u | ^{2} dx-\frac{1}{2} \int_{\Omega} | \nabla \bar v | ^{2} dx\geq \int_\Omega \nabla \bar v.\nabla (\bar u-\bar v)\, dx.
\end{equation}
Thus, by (\ref{ineq2}) and (\ref{ineq22}) we obtain that  
\begin{equation}\label{ineq1}
\frac{1}{2} \int_{\Omega} | \nabla \bar v | ^{2} dx-\frac{1}{2} \int_{\Omega} | \nabla \bar u | ^{2} dx= \int_\Omega \nabla \bar v.\nabla (\bar v-\bar u)\, dx.
\end{equation}
This indeed implies that 
\[\int_{\Omega} | \nabla \bar v-\nabla \bar u | ^{2} dx=0,\]
from which we obtain  $\bar v=\bar u.$ This completes the proof. \hfill $\square$\\

We shall use Theorem \ref{con2} to prove our main results in Theorems \ref{main1} and \ref{main2}. The convex closed subset $K$ of $V$ required in Theorem \ref{main2} is defined as follows  
 \[K(r) :=\big \{ u \in V : \| u \| _{W^{2, n}(\Omega)} \leq r \big  \}, \] for some $r > 0$ to be determined later. Also, the convex set $K$ required in the proof of Theorem \ref{main1} consists of all  non-negative functions in $K(r)$ for some $r>0.$ \\

To apply theorem \ref{con2}, we shall need to verify both conditions $(i)$ and $(ii)$ in this Theorem.  To verify condition  $(i)$  in Theorem \ref{main1} we simply find a minimizer of $I_K$ for  
some weakly compact and convex subset   $K $ of $V,$ and in Theorem \ref{main2} we shall make use of the abstract Theorem \ref{critical theorem} to find a sequence of solutions. However, condition $(ii)$ in Theorem \ref{con2} seems to be rather identical for both Theorems \ref{main1} and \ref{main2}.  Let us first proceed with condition $(ii)$ in Theorem \ref{con2}.  In fact, our plan is to show that if $\overline{u}\in K(r)$ then  for appropriate choices of $r,$ there exists $v \in K(r) $ such that $D \Psi(v) = D \Phi(\overline{u})$.
 We shall do this in a few lemmas. 
\begin{lemma} \label{3.2}
Assume that $1 < q < 2 < p$. Let $d_1$ and $d_2$ be the the best constants in the imbeddings  $W^{2,n}(\Omega)\hookrightarrow  L^{n(p-1)} (\Omega)$ and $W^{2,n}(\Omega)\hookrightarrow  L^{n(q-1)} (\Omega)$, respectively. Then
\begin{align*}
\| D\Phi (u) \|_{L^{n}(\Omega)} \leq   C_{1} r^{p -1} + \mu C_{2} r^{q -1}, \qquad \forall u \in K(r),
\end{align*} 
where   $C_1=d_1^{p-1}$ and $C_2=d_2^{q-1}.$ 
\end{lemma}

\noindent
\textbf{Proof.}
By definition of $D\Phi (u)$ we have
\begin{align*}
\| D\Phi (u) \|_{L^{n}(\Omega)} &= \big \| u | u| ^{p - 2} +\mu u | u | ^{q - 2} \big \|_{L^{n}(\Omega)} \leq \big \| u | u| ^{p - 2}\big \|_{L^{n}(\Omega)} + \mu \big \| u | u | ^{q - 2} \big \|_{L^{n}(\Omega)} \\
&\leq \| u \|_{L^{n(p -1)}(\Omega)}^{(p -1)} +  \mu \| u \|_{L^{n(q -1)}(\Omega)}^{(q -1)}.
\end{align*} 
By Theorem \ref{imbedding theorem} the space $W^{2, n}(\Omega)$ is compactly imbedded in $L^{n(p -1)}$ and $L^{n(q -1)}$. Thus,
\begin{align*}
\| D\Phi  (u) \|_{L^{n}(\Omega)}  \leq   C_{1} \| u \|_{W^{2, n}(\Omega)}^{p -1} + \mu C_{2} \| u \|_{W^{2, n}(\Omega)}^{q -1}.
\end{align*} 
It follows from $u \in K(r)$ that
\begin{align*}
\| D\Phi  (u) \|_{L^{n}(\Omega)}  \leq C_{1} r^{p -1} + \mu C_{2} r^{q -1} ,
\end{align*} 
as desired.
\hfill$\square$\\

By a straightforward  computation one can easily deduce the following result.
\begin{lemma} \label{300}
Let $1 < q < 2 < p$.  Assume that $C_1$ and $C_2$ are given in Lemma \ref{3.2}. Then there exists  $\mu ^{*} > 0$ with the following properties.
\begin{enumerate}
\item For each $\mu \in (0,\mu ^{*})$, there exist positive numbers $r_{1},  r_{2} \in \mathbb{R}$ with $r_{1} < r_{2}$ such that  $r \in  [r_{1}, r_{2}]$ if and only if   $C_{1} r^{p -1} + \mu C_{2} r^{q -1}\leq r.$
\item  For $\mu=\mu^*,$ there exists one and only one  $r>0$ such that $C_{1} r^{p -1} + \mu C_{2} r^{q -1}= r.$
\item For $\mu> \mu^*,$ there is no $r>0$ such that $C_{1} r^{p -1} + \mu C_{2} r^{q -1}= r.$
\end{enumerate}
\end{lemma}

\begin{remark}\label{pe}
Since the Sobolev space $W^{2, n}(\Omega)$ is compactly embedded into $L^{p}(\Omega)$, we obtain that \[V \cap W^{2, n}(\Omega)=H_0^1(\Omega) \cap W^{2, n}(\Omega).\] It also follows from  Corollary \ref{p-est} that $u \to \|\Delta u\|_{L^n(\Omega)}$ is an equivalent norm on  $H_0^1(\Omega) \cap W^{2, n}(\Omega)$.  For the rest of the paper, we shall then  consider this norm, i.e.,  for each  $u \in H_0^1(\Omega) \cap W^{2, n}(\Omega),$
\[\|u\|_{W^{2,n}(\Omega)}=\|\Delta u\|_{L^n(\Omega)}.\]
\end{remark}

We are now in the position to state the following result addressing condition $(ii)$ in Theorem \ref{con2}. 
\begin{lemma} \label{3.3}
Let  $1 < q < 2 < p$. Assume that $\mu ^{*} > 0$ is given in Lemma \ref{300} and  $\mu \in (0, \mu^*).$ Let $r_1, r_2$ be given in part 1) of  Lemma \ref{300}. Then   for each $r \in  [r_{1}, r_{2}]$  and   each $\overline{u}\in K(r)$ there exists $v \in K(r)$ such that 
 \[-\Delta v=\overline{u} | \overline{u}| ^{p - 2} +\mu \overline{u} | \overline{u} | ^{q - 2} .\]
\end{lemma} 

\noindent
\textbf{Proof.}
By standard methods we see that there exists   $v \in H_0^1(\Omega)$ which satisfies 
\begin{align} \label{4m}
-\Delta v =D\Phi (\overline{u}) = \overline{u} | \overline{u}| ^{p - 2} +\mu \overline{u} | \overline{u} | ^{q - 2} .
\end{align} 
in a weak sense. Since the right hand side is an element in $L^n(\Omega),$ it follows from the standard regularity results that $v \in W^{2,n}(\Omega) \cap H_0^1(\Omega)$ and $(\ref{4m})$ holds pointwise. Therefore,
\[\|\Delta v\|_{L^n(\Omega)} =\|D\Phi (\overline{u})\|_{L^n(\Omega)}. \]
Thus, by Remark \ref{pe} we have that 
\[\| v\|_{W^{2,n}(\Omega)} =\|D\Phi (\overline{u})\|_{L^n(\Omega)}. \]
This together with  Lemma \ref{3.2} yield that 
\begin{align*}
\| v \|_{W^{2, n}(\Omega)}  \leq C_{1} r^{p -1} + \mu C_{2} r^{q -1},
\end{align*}
By Lemma \ref{300},  for each $r \in  [r_{1}, r_{2}]$ we have that  $C_{1} r^{p -1} + \mu C_{2} r^{q -1}\leq r.$ Therefore,
\[\| v \|_{W^{2, n}(\Omega)}  \leq C_{1} r^{p -1} + \mu C_{2} r^{q -1}\leq r,\]
as desired. \hfill$\square$\\

\noindent
\textbf{Proof of Theorem \ref{main1}.}
Let $\mu^*$ be as in Lemma \ref{3.3} and $\mu \in (0, \mu^*).$ Also, let  $r_1$ and $r_2$ be as in Lemma \ref{3.3} and define 
\[K:=\big \{u \in K(r_2); \, \, \, u(x) \geq 0 \text{ a.e. } x \in \Omega\big \}.\]
{\it Step 1.}  We show that there exists  $\overline{u} \in K$ such that $I_K(\overline{u}) = \inf _{u \in V}I_K(u)$. Then by Proposition \ref{minimum}, we conclude that  $\overline{u}$  is a critical point of $I_K.$\\
Set $\eta := \inf _{u \in V}I_K(u)$. So by definition of $\Psi_K$ for every $u \notin K$, we have $I_K(u) = +\infty$ and therefore $\eta = \inf _{u \in K}I_K(u)$. On the other hand, by Theorem \ref{imbedding theorem}, the Sobolev space $W^{2, n}(\Omega)$ is compactly embedded in $L^{t}(\Omega)$ for all $t$, it then follows that for every $u \in K$ 
\begin{align*}
\Phi (u) & = \dfrac{1}{p} \int_{\Omega} | u | ^{p} dx+ \dfrac{\mu}{q} \int_{\Omega} | u | ^{q} dx \\
 &\leq c_{1} \| u\|_{W^{2, n}(\Omega)}^{p} + c_{2} \| u\|_{W^{2, n}(\Omega)}^{q} \leq  c_{1} r^{p} + c_{2} r^{q},
\end{align*}
for some positive constants $c_{1}$ and  $c_{2}$. Since $\Psi (u)$ is nonnegative, we have
 \begin{eqnarray*}
  I_K(u):= \Psi_K(u)-\Phi(u) \geq -\left(c_{1} r^{p} + c_{2} r^{q}\right),\qquad \forall u \in K.
 \end{eqnarray*}
So $\mu > -\infty$. 
Now, suppose that $\{u_{n}\}$ is a sequence in $V$ such that $I_K(u_{n})\rightarrow \eta$. So the sequence  $\{I_K(u_{n})\}$ is bounded and we can conclude by definition of $I_K$ that the sequence $\{u_{n}\}$ is bounded in $W^{2, n}(\Omega)$. Using standard results in Sobolev spaces, after passing to a subsequence if necessary, there exists $\overline{u} \in K$ such that $u_n \rightharpoonup \overline{u} $ weakly in $W^{2, n}(\Omega)$ and strongly in $V.$ Therefore, $I_K(u_{n})\rightarrow I_K(\overline{u})$. So, $I_K(\overline{u}) = \eta=\inf _{u \in V}I_K(u),$ and the proof of {\it Step 1} is complete.\\

{\it Step 2.} In this step we show that there exists $v\in K$ such that  $-\Delta v=\overline{u} | \overline{u}| ^{p - 2} +\mu \overline{u} | \overline{u} | ^{q - 2} .$ By  Lemma \ref{3.3} together with the fact that $\overline{u} \in K(r_2)$ we obtain that $v\in K(r_2).$ To show that $v\in K,$ we shall need to verify that $v$ is non-negative almost every where.  But, this is a simple consequence 
of the maximum principle and the fact  that $-\Delta v=\overline{u} | \overline{u}| ^{p - 2} +\mu \overline{u} | \overline{u} | ^{q - 2} \geq 0.$ \\

It now follows from Theorem \ref{con2} together with 
{\it Step 1} and 
{\it Step 2} that $\overline{u}$ is a solution of the problem (\ref{eq}).  To complete the proof we shall show that $\overline{u}$ is non-trivial by proving that 
 $I_{K}(\overline{u}) = \inf_{u \in V} I_{K}(u) < 0$. 
\\
Take $e \in K$. For $t\in [0,1],$ we have that  $te \in K$ and therefore
\begin{align*}
I_{K}(te) & = \frac{1}{2} \int_{\Omega} | \nabla te | ^{2} dx - \frac{1}{p} \int_{\Omega} | te | ^{p} dx - \frac{\mu}{q} \int_{\Omega} | te | ^{q} dx
\\ 
& = t^{q}\left( \frac{ t^{2 - q}}{2} \int_{\Omega} | \nabla e | ^{2} dx -\frac{ t^{p - q}}{p} \int_{\Omega} | e | ^{p} dx - \frac{\mu}{q} \int_{\Omega} | e | ^{q} dx\right).
\end{align*}
Since $1 < q < 2 < p$, $I_{K}(te)$ is negative for $t$ sufficiently small. Thus,  we can conclude that $I_{K}(\overline{u}) = \inf_{u \in V} I_{K}(u) < 0$. Thus, $\overline{u}$ is a non-trivial and non-negative solution of (\ref{eq}). Finally,  it follows from the strong maximum principle that $\overline u>0$ on $\Omega$. 
\hfill$\square$\\

\textbf{Proof of Theorem \ref{main2}.}
Let $\mu^*$ be as in Lemma \ref{3.3} and $\mu \in (0, \mu^*).$ Also, let  $r_1$ and $r_2$ be as in Lemma \ref{3.3} and define $K=K(r_2).$   We first show  that the functional $I_K$ has infinitely many distinct critical points.  To do this,  we shall employ Theorem \ref{critical theorem}.  It is obvious that the function
$\Phi$ is even and continuously differentiable. Also $\Psi_K$ is a proper, convex and lower semi-continuous even function. So ${(H)}$ is satisfied. We now  verify ${(PS)}$. If $I_K(u_{n})\rightarrow c$ for some $c\in \mathbb{R}$, by definition of $I_K$ we can conclude that $\{u_{n}\}$ is bounded in $W^{2, n}(\Omega)$. Going if necessary to a subsequence, there exists some $\overline{u} \in W^{2,n}(\Omega) \cap H_0^1(\Omega)$ such that $u_{n}\rightharpoonup \overline{u}$ weakly in $W^{2, n}(\Omega).$  Due to the compact imbeddings of $W^{2, n}(\Omega)\hookrightarrow H_0^1(\Omega)$ and $W^{2, n}(\Omega)\hookrightarrow L^{p}(\Omega)$ we obtain that  $u_{n}\rightarrow \overline{u}$ in $V$ strongly. 

For each  $k\in \mathbb{N},$ considering the definition of $\Gamma_k$ in (\ref{Gam}), we define
\begin{align*}
c_{k} = \inf_{A\in \Gamma_{k}} \sup_{u \in A} I(u).
\end{align*}
We shall now prove  that $-\infty < c_{k} < 0$ for all $k \in \mathbb{N}.$
To this, let us denote by $\lambda_{j}$ the j-th eigenvalue of $-\Delta$ on $H_0^1(\Omega)$ (counted according to its multiplicity) and by $e_{j}$ a corresponding eigenfunction satisfying $\int_{\Omega} \nabla e_{i}.\nabla  e_{j}\, dx = \delta_{ij}$.  As in the proof of Theorem \ref{main1}, we have that  $I_K$ is bounded below. Thus $ c_{k} > -\infty$ for each $k \in \mathbb{N}.$ Let 
\begin{align*}
A:= \big \{u = \alpha_{1}e_{1} + ... + \alpha_{k}e_{k}  : \| u \|_{H_0^1(\Omega)}^{2} = \alpha_{1}^{2} + ... + \alpha_{k}^{2} = \rho^{2}\big \}, 
\end{align*}
for small $\rho > 0$ to be determined. Then $A \in \Gamma_{k}$ because $\gamma(A) = k$ by Proposition \ref{genus}. Since $A$ is finite dimensional, all norms are equivalent on $A$. Thus,  we can choose  $\rho$ small enough  so  that $A \subseteq K$. Also there exist positive  constants $c_{1}$, $c_{2}$ such that $\| u \|_{L^{p}(\Omega)} > c_{1} \| u \|_{H_0^1(\Omega)}$ and $\| u \|_{L^{q}(\Omega)} > c_{2} \| u \|_{H_0^1(\Omega)}$ for all $u \in A$. Therefore,
\begin{align*}
 I_K(u) &= \frac{1}{2}\| u \|_{H_{0}^{1}(\Omega)}^{2} - \dfrac{1}{p}\| u \|_{L^{p}(\Omega)}^{p} - \dfrac{\mu}{q}\| u \|_{L^{q}(\Omega)}^{q} \\
 & \leq  \dfrac{1}{2} \rho^{2} - \dfrac{1}{p}c_{1}^{p} \rho^{p} - \dfrac{\mu}{q} c_{2}^{q} \rho^{q} = \rho^{q} (\dfrac{1}{2}\rho^{2 - q} - \dfrac{1}{p}c_{1}^{p} \rho^{p - q} - \dfrac{\mu}{q} c_{2}^{q}).
\end{align*}
Now we can choose $\rho$ small enough such that $ I_K(u) \leq \rho^{q} (\dfrac{1}{2}\rho^{2 - q} - \dfrac{1}{p}c_{1}^{p} \rho^{p - q} - \dfrac{\mu}{q} c_{2}^{q}) < 0$ for every  $u \in A$. It then follows that $c_{k} < 0$. Thus, by Theorem \ref{critical theorem}, $I_K$ has has a sequence of distict critical points $\{u_k\}_{k \in \N}$ by means of Definition \ref{dd}. Also, by Lemma \ref{3.3}, for each critical point $u_k$ of $I_K$ there exists $v_k \in K$ such that $-\Delta v_k=D \Phi(u_k).$ It now follows from Theorem \ref{con2} that $\{u_k\}$ is a sequence of distinct solutions of (\ref{eq}) such that $I_K(u_k)<0$ for each $k \in \N.$ This completes the proof.
\hfill$\square$\\

It is evident that   Theorem {\ref{main2}} can be easily extended to $p-$laplacian problems similar to the problem  (\ref{eq}). Indeed, consider
\begin{equation}\label{conpp}
\left \{
\begin{array}{ll}
-\Delta_p u =|u|^{r-2} u+\mu |u|^{q-2}u, &   x \in \Omega\\
u=0, &  x \in \partial \Omega
\end{array}
\right.
\end{equation}
By using a similar   argument as in the proof of Theorem \ref{main2} one can prove that, if  $1 < q < p < r$ then there exists $\mu ^{*} > 0$ such that for each $\mu \in (0,\mu ^{*})$ problem (\ref{conpp}) has infinitely many distinct nontrivial solutions with a negative energy.  
In our forthcoming project,  we are investigating  the existence of two positive solutions in Theorem \ref{main1} rather than just one.  We are also    extending both theorems \ref{main1} and \ref{main2} to the fractional laplacian case via the method proposed in this manuscript.


\begin{thebibliography}{99}
  
\bibitem {Adimurthi} F. Pacella, S. L. Yadava, \emph{On the number of positive solutions of some
semilinear Dirichlet problems in a ball}, Differential Integral Equations 10 (1997), no. 6, 1157-1170.
  
 
\bibitem {Azorero}  A. Ambrosetti, J. Garcia Azorero, I. Peral, \emph{Multiplicity results for some nonlinear
elliptic equations}, J. Funct. Anal. 137 (1996) 219-242.
  
\bibitem {Ambro} A. Ambrosetti, H. Brezis,  G. Cerami, \emph{Combined effects of concave and convex
nonlinearities in some elliptic problems}, J. Funct. Anal., 122(2):519-543, 1994.

\bibitem{Willem} T. Bartsch and M. Willem, \emph{On an elliptic equation with concave and convex nonlinearities},
Proc. Amer. Math. Soc., 123(11):3555-3561, 1995.

\bibitem{ACL} C. Cowan, A. Moameni,  \emph{A new variational principle, convexity and supercritical Neumann problems.}
To appear in Trans. Amer. Math. Soc. (2017).

\bibitem{Damascelli} L. Damascelli, M. Grossi, F. Pacella, \emph{Qualitative properties of positive solutions of
semilinear elliptic equations in symmetric domains via the maximum principle}, Ann. Inst. H.  Poincar\'e Anal. Non Lin\'eaire 16 (1999) 631-652 .


\bibitem{Ek-Te} I.  Ekeland, R. Temam, \emph{Convex analysis and variational problems,} American Elsevier Publishing Co., Inc., New York, (1976).

\bibitem{Garcia} J. Garcia Azorero, I. Peral Alonso, \emph{Multiplicity of solutions for elliptic problems with critical
exponent or with a nonsymmetric term}, Trans. Amer. Math. Soc., 323 (1991), 877-895.


\bibitem{G-T} D. Gilbarg,  N.S. Trudinger, \emph{Elliptic partial differential equations of second order.} Reprint of the (1998) edition. Classics in Mathematics. Springer-Verlag, Berlin, (2001).

 \bibitem{Kuratowski} K. Kuratowski, \emph{Topologie I}, PWN, Warszawa,1958.\label{r18}
	
\bibitem{Mo} A. Moameni, \emph{New variational principles of symmetric boundary value problems}, J.  Convex Anal., 24 (2017), no. 2, 365-381.



\bibitem{Mo2}  A. Moameni,  \emph{Non-convex self-dual Lagrangians: New variational principles of symmetric boundary value problems}, J. Func. Anal., 260 (2011), 2674-2715.



 \bibitem{Mo5} A. Moameni,  \emph{A variational principle for partial differential equations with a hint
of convexity}, Submitted,  	arXiv:1705.08348.

\bibitem{Mo1}  A. Moameni, \emph{A variational principle for problems in nonlinear Analysis}, In preparation. 

\bibitem{Ouyang} T. Ouyang, J. Shi, \emph{Exact multiplicity of positive solutions for a class of semilinear
problem II}, J. Differential Equations 158 (1999) 94-151.

\bibitem{Rabinowitz 1} P. H. Rabinowitz, \emph{Variational methods of nonlinear eigenvalue problems}, Proc. Sym. on Eigenvalues of nonlinear problems, Edizionicremonese, Rome, 1974, p. 143-195. 

\bibitem{Rabinowitz 2} P. H. Rabinowitz, \emph{some aspects of critical point theory}, MRC Tech, Rep. 2465, Madison, Wisconsin, 1983. 

\bibitem{szulkin}  A. Szulkin, \emph{ Minimax principles for lower semicontinuous functions and applications to nonlinear boundary value problems},  Ann. Inst. H. Poincar\'e Anal. Non Lin\'eaire 3 (1986), no. 2, 77-109.\label{r19}

\bibitem{Tang} M. Tang, \emph{Exact multiplicity for semilinear elliptic Dirichlet problems involving concave and convex nonlinearities}, Proc. Roy. Soc. Edinburgh Sect. A, 133 (2003) 705-717.


\bibitem{Wang} Z. Q. Wang, \emph{Nonlinear boundary value problems with concave nonlinearities near the origin}, 
NoDEA Nonlinear Differential Equations Appl., 8 (2001), 1, 15-33.

\bibitem{minimax book} M. Willem, \emph{minimax Theorems}, progress in nonlinear differential equations and their applications, v 24, 1996.

\end{thebibliography}
\end{document}